\documentclass[a4paper]{amsart}

\usepackage{graphicx}
\usepackage{amsmath}
\usepackage{amsthm}
\usepackage{amssymb}
\usepackage[all,2cell]{xy}
\usepackage{dsfont}
\usepackage[colorlinks,linkcolor=blue,anchorcolor=red,citecolor=green]{hyperref}
\usepackage{amscd}
\usepackage{chemarrow}
\usepackage{color}

\UseAllTwocells
\SilentMatrices

\theoremstyle{plain}
\newtheorem{thm}{Theorem}[section]
\newtheorem*{thm*}{Theorem}
\newtheorem{prop}[thm]{Proposition}
\newtheorem{cor}[thm]{Corollary}
\newtheorem*{cor*}{Corollary}
\newtheorem{lem}[thm]{Lemma}
\theoremstyle{definition}
\newtheorem{defn}[thm]{Definition}
\newtheorem{exm}[thm]{Example}
\theoremstyle{remark}
\newtheorem{rmk}[thm]{\bf Remark}

\newcommand{\Ker}{\mathrm{Ker}}

\begin{document}
\title[]{a note on path $A_{\infty}$-algebras of positively graded quivers}
\author[]{Hao Su}
\thanks{Supported by the National Nature Science Foundation of China, no 11431010 and 11571329.}
\address{School of Mathematical Sciences \\
         University of Science and Technology of China \\
         Hefei, Anhui 230026 \\
         P. R. China}
\email{suhao@mail.ustc.edu.cn}
\dedicatory{}
\commby{}
\date{\today}
\begin{abstract}
  Let $A$ be a path $A_{\infty}$-algebra over a positively graded quiver $Q$. It is proved that the derived category of $A$ is triangulated equivalent to the derived category of $kQ$, which is viewed as a dg algebra with trivial differential. The main technique used in the proof is $Koszul$ $duality$.
\end{abstract}
\maketitle

\section{Introduction}

    $A_{\infty}$-algebras, invented by J. stasheff \cite{Stasheff1963}, are natural generalization of associative algebras. In this note, we discuss a specific class of $A_{\infty}$-algebras obtained from quivers, which we call path $A_{\infty}$-algebras.
    Throughout $k$ is a field. Let $Q$ be a finite positively graded quiver and $kQ$ its path algebra. By a path $A_{\infty}$-algebra over $Q$ it is meant a minimal strictly unital $A_{\infty}$-algebra over $kQ_{0}$ whose underlying associative graded algebra is $kQ$. We say that an augmented $A_{\infty}$-algebra $A$ is quasi-equivalent to a dg algebra $B$ if the enveloping algebra $U$ of $A$ is quasi-equivalent to $B$ as dg algebras (see \cite[section 7]{Keller1994}).
    Our main result is as follows.

    \begin{thm*}
        Let $A$ be a path $A_{\infty}$-algebra over a finite positively graded quiver $Q$. Then $A$ is quasi-equivalent to $kQ$ which is viewed as a dg algebra with trivial differential. In particular, the derived category of $A$ is triangulated equivalent to the derived category of $kQ$.
    \end{thm*}

    The motivation for this result is the following consequence, which concerns the uniqueness of triangulated structures on a given additive category. Let $\mathcal{T}$ be a compactly generated (k-linear) algebraic triangulated category and $T$ a compact generator of $\mathcal{T}$. Let $A_{\mathcal{T}}$ denote the graded algebra with underlying graded vector space $\bigoplus\limits_{n \in \mathbb{Z}}Hom_{\mathcal{T}}(T, \sum^{n}T)$ and whose multiplication is induced by the composition of morphisms in $\mathcal{T}$.

    \begin{cor*}
        Keep the above notations and assumptions. If $A_{\mathcal{T}}$ is isomorphic to the path algebra of a finite positively graded quiver, then $\mathcal{T}$ is triangulated equivalent to the derived category of $A_{\mathcal{T}}$, which is viewed as a dg algebra with trivial differential.
    \end{cor*}

\section{$A_{\infty}$-algebras}
    In this section we review some of the standard facts on $A_{\infty}$-algebras. Most of the following definitions and results come from \cite{Lefevre-Hasegawa2003}. A good introduction on this subject is Keller's paper \cite{Keller1999}. Let $R\colon= \prod_{i=1}^{n}k$ be a finite product of $k$. For simplicity  we use the unadorned $\otimes$ to denote the tensor product over $R$. Here is the definition:

    \begin{defn}\cite[Definition 1.2.1.1]{Lefevre-Hasegawa2003}\label{def2.1}
        An \emph{$A_{\infty}$-algebra over $R$} is a $\mathds{Z}$-graded $R$-bimodule $A=\bigoplus\limits_{n\in \mathds{Z}}A^{n}$ (we do not require that the left and right $R$-module structures coincide) endowed with a family of homomorphisms of graded $R$-bimodules $m_{n} \colon A^{\otimes n} \to A , n \geq 1$, of degree $2-n$ satisfying the following identities (Stasheff identities):
        $$
            \sum\limits_{r+s+t=n} (-1)^{rs+t}m_{r+1+t}(id_{A}^{\otimes r} \otimes m_{s} \otimes id_{A}^{\otimes t})=0,
        $$
        where $r,t \geq 0$ and $s\geq 1$.
    \end{defn}

    \begin{rmk}
        Our definition of $A_{\infty}$-algebras is more general than the usual definition (which requires that left and right graded module structure of $A$ over $R$ coincide, see for example \cite[Definition 2.1]{Lunts10}), but fits in the general framework of lefevre in \cite{Lefevre-Hasegawa2003}. He deals with $A_{\infty}$-algebras in a monoidal abelian category $(\mathcal{C},\odot)$, which is assumed to be $k$-linear semisimple cocomplete with exact filtered colimits such that for any object $M\in \mathcal{C}$, the functors $M \odot ?$ and $? \odot M$ are exact and commute with filtered colimits. Our $A_{\infty}$-algebras in Definition \ref{def2.1} are exact the $A_{\infty}$ algebra in the category of graded $R$-bimodules.
    \end{rmk}

    \begin{defn}\cite[Definition 1.2.1.2]{Lefevre-Hasegawa2003}
        A $morphism$ of $A_{\infty}$-algebras $f\colon A \to A'$ is a family of homomorphisms of graded $R$-bimodules $f_{n}\colon A^{\otimes n} \to A'$, $n \geq 1$ of degree $1-n$ satisfying the following identities:
        $$
            \sum_{r+s+t=n}(-1)^{rs+t}f_{r+1+t}(id^{\otimes r} \otimes m_{s} \otimes id^{\otimes t})
        = \sum_{\substack{1\leq p \leq n \\ i_{1}+\ldots+i_{p}=n}} (-1)^{\omega}m_{p}(f_{i_{1}} \otimes \cdots \otimes f_{i_{p}}),
        $$
        where $s\geq 1;r,t \geq 0$ and $\omega = \sum\limits_{1\leq x < y \leq p} i_{x}i_{y} + i_{1}(p-1) + i_{2}(p-2)+ \ldots + i_{p-1}$.
    \end{defn}

    By definition, $f_{1}$ is a morphism of complexes. We say that an $A_{\infty}$-morphism $f$ is a $quasi$-$isomorphism$ if $f_{1}$ is a quasi-isomorphism of complexes.

    An $A_{\infty}$-morphism $f$ is said to be \emph{strict} if $f_{i}=0$ for $i >1$. An $A_{\infty}$-algebra $A$ is \emph{strictly unital} if there is a strict morphism $\eta\colon R \to A$ of $A_{\infty}$-algebras (which is called the \emph{unit}) such that $m_{n}(id_{A} \otimes \cdots  \otimes id_{A} \otimes \eta \otimes id_{A} \otimes \cdots \otimes id_A) = 0$ for $n \neq 2$ and $m_{2}(\eta \otimes id_{A})= id_{A} = m_{2}( id_{A} \otimes \eta)$. Morphisms between strictly unital $A_{\infty}$-algebras are required to preserve units. A strictly unital $A_{\infty}$-algebra is \emph{augmented} if it is endowed with a strict morphism $\epsilon \colon A \to R$ of strictly unital $A_{\infty}$-algebras (which is called an \emph{augmentation}) such that $\epsilon \circ \eta = id_{R}$. The graded $R$-bimodule $\bar{A}=\Ker \epsilon$ is called the \emph{augmentation ideal} of $A$. Morphisms between augmented $A_{\infty}$-algebras are also required to preserve augmentation. It is easy to see the restrictions of  $m_{n}{\prime s}$ on $\bar{A}$ make $\bar{A}$ an $A_{\infty}$-algebra which is not unital in general.

    Given a graded $R$-bimodule $M$, we denote by $sM$ the graded $R$-bimodule which has the same underlying bimodule structure as $M$ but with a new grading given by $(sM)^{n} = M^{n+1}$. We use the notation `$s$' to denote the map $M \to sM$ of degree $-1$ which is the identity map when forgetting the grading.  Let $(A,m_{n})$ be an augmented $A_{\infty}$-algebra. Define $d_{n}\colon (s\bar{A})^{\otimes_{R} n} \rightarrow s\bar{A}$ by $d_{n}=-s \circ m_{n} \circ (s^{-1})^{\otimes_{R} n}$, then $\bigoplus_{n > 0} d_{n}$ induces a $coderivation$ $d$ of degree $1$ on the tensor coalgebra $T_{R}(s\bar{A})$. In fact, it is proved in \cite{Stasheff1963} that $d$ is a differential on $T_{R}(s\bar{A})$ and $(T_{R}(s\bar{A}),d)$ is an augmented dg coalgebra, which is called the $bar$ $construction$ for $A$ and denoted by $BA$.

    $A_{\infty}$-algebras are closely related to dg algebras. On the one hand, it is easy to see that a dg algebra over $R$ is just an $A_{\infty}$-algebra with $m_{n}$'s vanishing for $n \geq 3$.  On the other hand, as the following proposition shows, any augmented $A_{\infty}$-algebra is $A_{\infty}$-quasi-isomorphic to a dg algebra.

    \begin{prop}[{\cite[lemma 2.3.4.3]{Lefevre-Hasegawa2003}}]\label{p2.4}
        For an augmented $A_{\infty}$-algebra $A$, there is an augmented dg algebra $U$ (called the \emph{enveloping algebra} of $A$) and  a natural $A_{\infty}$-quasi-isomorphism $f\colon A \to U$ which preserves augmentation.
    \end{prop}

    \begin{rmk}
        There is a construction dual to bar construction, named by $cobar$ construction, which is a functor from the category of augmented dg coalgebras to the category of augmented dg algebras. Let $C$ be an augmented dg coalgebra, the cobar construction of $C$ is denoted by $\Omega C$. The enveloping algebra $U$ defined in Proposition \ref{p2.4} of an augmented $A_{\infty}$-algebra $A$ is actually defined by $\Omega B A$.
    \end{rmk}

    An $A_{\infty}$-algebra $(A,m_{i})$ is said to be $minimal$ if $m_{1}=0$. If $A$ is minimal, then $(A,m_{2})$ is an associative graded algebra. Let $A$ be an $A_{\infty}$-algebra, a minimal model of $A$ is an minimal $A_{\infty}$-algebra $A'$ together with an $A_{\infty}$-quasi-isomorphism $f: A' \to A$. The following lemma shows that an augmented $A_{\infty}$-algebra admits a minimal model which is still augmented.

    \begin{lem}\label{minimal-model}
        Let $A$ be an augmented $A_{\infty}$-algebra over $R$. Then there is a minimal augmented $A_{\infty}$-algebra $A'$ and an $A_{\infty}$-quasi-isomorphism
        $f\colon A' \to A$ which is strictly unital and preserves augmentation.
    \end{lem}

    \begin{proof}
        Since $A$ is an augmented $A_{\infty}$-algebra, $A$ is a direct sum of two $A_{\infty}$-algebras $R$ and $\bar{A}$. By {\cite[Proposition 3.2.4.1]{Lefevre-Hasegawa2003}} (see also \cite{Kad1982}), there is a minimal $A_{\infty}$-algebra $\bar{A}'$ and an $A_{\infty}$-quasi-isomorphism $f\colon \bar{A}' \to \bar{A}$ ($\bar{A}'$ is a minimal model of $\bar{A}$). Then the augmented $A_{\infty}$-algebra $A'=R\oplus \bar{A}'$  satisfies the desired properties.
    \end{proof}

    \begin{defn}[{\cite[Definition 2.3.1.1]{Lefevre-Hasegawa2003}}]
        Let $A$ be an $A_{\infty}$-algebra over $R$. A (right) $A_{\infty}$-module $M$ over $A$ is a graded right $R$-module $M=\bigoplus\limits_{n\in \mathds{Z}}M^{n}$ endowed with a family of homomorphisms of graded right $R$-modules $m^{M}_{n} \colon M\otimes A^{\otimes n-1} \to M , n \geq 1$, of degree $2-n$ satisfying the following identities (Stasheff identities):
        $$
            \sum\limits_{r+s+t=n} (-1)^{rs+t}m_{r+1+t}(id_{A}^{\otimes r} \otimes m_{s} \otimes id_{A}^{\otimes t})=0,
        $$
        where $r,t \geq 0$ and $s\geq 1$.
    \end{defn}

    Let $A$ be a strictly unital $A_{\infty}$-algebra, an $A_{\infty}$-module $M$ over $A$ is strictly unital if $$m_{i}^{M}(id_{M} \otimes id_{A},\ldots,id_{A} \otimes \eta \otimes id_{A},\ldots,id_{A})=0$$
    for $i \geq 3$ and $m_{2}^{M}(id_{M} \otimes \eta)=id_{M}$. For more details on $A_{\infty}$-modules we refer to \cite[Section 2.3]{Lefevre-Hasegawa2003} and \cite[Section 4]{Keller1999}. Let $A$ be an augmented $A_{\infty}$-algebra over $R$ and $Mod_{\infty}(A)$ be the the category of strictly unital $A_\infty$-modules over $A$ with strictly unital
    $A_{\infty}$-morphisms as morphisms. The \emph{derived category} $\mathcal{D}(A)$ is the category obtained from $Mod_{\infty}(A)$ by formally inverting all $A_{\infty}$-quasi-isomorphisms. The category $\mathcal{D}(A)$ is a triangulated category whose suspension functor is the shift functor $[1]$.

    \begin{rmk}\label{rmk2.8}
        If $A$ is a dg algebra which is unital over $R$, a (right) unital dg module $M$ over $A$ can be viewed as an strictly unital $A_{\infty}$-module over $A$ with $m_{n}^{M}=0$ for $n \geq 3$.  By \cite[Lemme 4.1.3.8]{Lefevre-Hasegawa2003}, the derived category of dg $A$-modules is canonically equivalent to the derived category of strictly unital $A_{\infty}$-modules over A. We will identify these two derived categories.
        Let $M,N$ be two right dg modules over $A$. The complex (over $k$) $\mathcal{H}om_{A}(M,N)$ of homomorphisms from $M$ to $N$ over $A$ is defined as follows. The component  $\mathcal{H}om^{i}_{A}(M,N)$ consists of all homogeneous maps of right $R$-modules $f\colon M \to N$ of degree $i$ such that $f(ma)=f(m)a$ for all $A\in A$ and $m\in M$.  The differential $d$ in $ \mathcal{H}om_{A}(M,N)$ is defined by $d(f)=d_{N}\circ f -(-1)^{|f|} f \circ d_{M}$. We use $\mathcal{E}nd_{A}(M)$ denote $\mathcal{H}om_{A}(M,M)$.
    \end{rmk}

    The following propositions will be used in the proof of Lemma \ref{l2}.

    \begin{prop}[{\cite[Theorem 4.1.2.4]{Lefevre-Hasegawa2003}}]\label{p2.10}
        Let $f\colon A' \to A$ be an $A_{\infty}$-mophism of augmented $A_{\infty}$-algebras. Then $f$ induces a triangulated functor $R_{f}\colon D_{\infty}(A) \to D_{\infty}(A')$. If $f$ is an $A_{\infty}$-quasi-isomorphism, then $R_{f}$ is a triangulated equivalence.
    \end{prop}

    \begin{prop}[{\cite[Proposition 3.3.1.7]{Lefevre-Hasegawa2003})}]\label{p2.11}
        Let $A$ be a strictly unital $A_\infty$-algebra, and $M$ a strictly unital $A_\infty$-module over $A$. Then there is a strictly unital minimal $A_\infty$-module $M'$ over $A$ together with a strictly unital quasi-isomorphism of  $A_\infty$-modules from $M'$ to $M$.
    \end{prop}

\section{Path $A_{\infty}$-algebra on a positively graded quiver}

    In this section, we will introduce the definition of path $A_{\infty}$-algebras over a finite positively graded quiver and prove the main theorem.

    \begin{defn}
        A $positively$ $graded$ $quiver$ $Q$ is a finite graded quiver whose arrows are in positive degrees and vertices are in degree $0$. We use $Q_{0}$ denote the set of vertices and $Q_{1}$ the set of arrows.
    \end{defn}

    Let $Q$ be a positively graded quiver and $kQ$ the path algebra. The subalgebra $kQ_0$ is clearly a separable commutative algebra over $k$. Denote by $u\colon kQ_0\to kQ$ the canonical embedding and $\varepsilon\colon kQ\to kQ_0$ the canonical projection.

    \begin{defn}\label{def3.2}
        A $path$ $A_{\infty}$-$algebra$ over a positively graded quiver $Q$ is defined to be a minimal, strictly unital $A_{\infty}$-algebra over $kQ_0$ with underlying associative graded algebra $kQ$ and the unit map $(u_1=u, 0,\ldots)$.
    \end{defn}

    \begin{rmk}
        Let $A=(kQ, m_2,\ldots)$ be a path $A_{\infty}$-algebra over a positively graded quiver $Q$. Since $A$ concentrated in nonnegative degrees and the $0$th degree is $kQ_{0}$, the canonical projection $(\varepsilon_{1}=\varepsilon, 0,\ldots)\colon A\to kQ_0$ makes $A$ an augmented $A_{\infty}$-algebra over $kQ_0$.
    \end{rmk}

    Here are two examples, the first one is a path $A_{\infty}$-algebra and the second is not.

    \begin{exm}\label{exm1}
        The graded quiver $Q$ is
        $$
        \xymatrix{
            4 \ar@/_2pc/[rrr]^{\tau} \ar[r]^{\gamma} & 3 \ar[r]^{\beta} & 2 \ar[r]^{\alpha} & 1
        }
        $$
        with $|\alpha|=|\beta| =|\gamma|=1$ , $|\tau|=2$ and $|e_{1}|,|e_{2}|,|e_{3}|,|e_{4}| = 0$.
        The $A_{\infty}$-structure on $kQ$ is defined as follows. $m_{n}=0$ for $n \neq 2,3$ and $m_{2}$ is the usual multiplication of $kQ$. For $m_{3}$, we define it on the basis elements by $m_{3}(x,y,z)=\tau$ if $(x,y,z)=(\alpha,\beta,\gamma)$ and $m_{3}(x,y,z)=0$ otherwise. This $A_{\infty}$ structure is well defined.
    \end{exm}
    The next example is constructed by \cite[Section 5]{Lu2004}.
    \begin{exm}
        The graded quiver $Q$ is
        $$
        \xymatrix{
            2 \ar[r]^{\alpha} & 1
        }
        $$
        with $|\alpha|=1$. Let $e= e_{1} - e_{2}$ then $e \cdot \alpha = \alpha = - \alpha \cdot e$ and $e^{2} = 1$. The $A_{\infty}$-structure on $kQ$ is defined as follows. $m_{n}=0$ for $n \neq 2,3$ and $m_{2}$ is the usual multiplication of $kQ$. For $m_{3}$, we defined it on basis elements by, $m_{3}(x,y,z)=\alpha$ if $(x,y,z)=(\alpha, e, \alpha)$ and $m_{3}(x,y,z)=0$ otherwise. Since $m_{3}(\alpha,e,\alpha)\neq 0$ which means $(kQ,m_{2},m_{3})$ is not strictly unital over $kQ_{0}$, so it is not a path $A_{\infty}$-algebras under definition \ref{def3.2}.
    \end{exm}

    The technique to prove the main theorem is Koszul duality for dg algebras, so here I give a brief introduction (see also  \cite[Section 10]{Keller1994}). Let $R$ be the direct product of $n$ copies of $k$ and $e_{1},\ldots,e_{n}$ be the standard basis of $R$ over $k$. Assume $A$ is an augmented dg algebra over $R$. Then $R$ is a dg bimodule via the augmentation map over $A$. We define a dg category $\mathcal{A}$ whose set of objects is ${e_{1},\ldots,e_{n}}$ and for any two objects $e_{i},e_{j}$ the homomorphism space $\mathcal{A}(e_{i},e_{j})$ is defined by $e_{j}Ae_{i}$. A right dg $A$-module can be viewed as a right dg module over $\mathcal{A}$ and vice versa. This implies a canonical triangulated equivalence between the two derived categories $\mathcal{D}(A)$ and $\mathcal{D}(\mathcal{A})$. Let $\{S_{1},\ldots,S_{n}\}$ be the collection of simple dg submodules of $R$ over $A$ such that $S_{i}e_{j}=0$ for $i\neq j$ and $S_{i}e_{i}\cong k$. Then $\mathcal{A}$ is an augmented dg category in the sence \cite[section 10.2]{Keller1994}. A $\emph{Koszul dual}$ $A^{\ast}$ of $A$ is defined as the endomorphism algebra $\mathcal{E}nd_{A}(P_{A})$ of a homotopically projective resolution $P_{A}$ of $R=S_{1}\oplus \ldots \oplus s_{n}$. Keller proved that the Koszul dual of $A$ does not rely on the choice of homotopically projective resolution of $R$ up to \emph{quasi-equivalence}. Here a quasi-equivalence $A\to B$ of dg algebras is a dg $A$-$B$-bimocule $X$ together with an element $x \in Z^{0}X$ such that the maps $ A \rightarrow X, a \mapsto ax$ and $B \rightarrow X, b\mapsto xb$
    are quasi-isomorphisms. A special case of quasi-equivalence if there is a quasi-isomorphism of dg algebra $A\to B$. If two dg algebras are quasi-equivalence, their derived categories are triangulated equivalence (\cite[Section 7.3]{Keller1994}).

    Keep the above assumptions and let $M$ be a right dg module over $A$. Let $p\emph{dim} M$ (resp. $i\emph{dim} M$) denote the $projective$ $dimension$ (resp. the $injective$ $dimension$) of $M$ (see \cite[Section 10.4]{Keller1994}). For a graded $k$-vector space $M$, let $D(M)$ denote the graded $k$-dual of $M$ which is defined by $D(M)^{n}= Hom_{k}(M_{R}^{-n},k)$. A graded $k$-vector space $M$ is called $locally$ $finite$ if each homogeneous component of $M$ has finite dimension. The following lemma is just an `algebra' version of \cite[Section 10.5 Lemma (the `symmetric' case)]{Keller1994}.

    \begin{lem}[{\cite[Section 10.5 Lemma (the `symmetric' case)]{Keller1994}}]\label{l2.5}
        Let $A$ be an augmented dg algebra over $R$.  Suppose $p\emph{dim} R_{A} < \infty$, $i\emph{dim} R_{A} < \infty$ and $D(A)$ belongs to the smallest triangulated subcategory of $\mathcal{D}(A)$ closed under direct sums and containing $R$. Then $\mathcal{E}nd_{A}(D(A))$ is quasi-equivalent to $A^{\ast\ast}$. We further assume $H^{\ast}(A)$ is locally finite,
        then $A$ is quasi-equivalent to $\mathcal{E}nd_{A}(D(A))$. So $A$ is quasi-equivalent to $A^{\ast\ast}$.
    \end{lem}

    Let $A$ be an augmented dg algebra over $R$ and $\pi$ denote the composition of the following morphisms $BA \xrightarrow{p} s\bar{A} \xrightarrow{s^{-1}} \bar{A} \xrightarrow{i} A$, where $BA$ is the bar construction of $A$, $p$ is the canonical projection and $i$ is the inclusion morphism. Define a new differential $d_{\pi}$ on $BA \otimes_{R} A$ as
    $$ d_{BA}  \otimes_{R} id_{A} + id_{BA} \otimes_{R} d_{A} + (id_{BA} \otimes_{R} \mu)(id_{BA} \otimes_{R} \pi \otimes_{R} id_{A})(\Delta_{BA} \otimes_{R} id_{A}),$$
    where $\mu$ is the multiplication of $A$ and $\Delta_{BA}$ is the comultiplication of $BA$. This new complex $(BA \otimes A, d_{\pi})$ is denoted by $BA \otimes_{\pi} A$ and is called the $augmented$ $bar$ $construction$ of $A$ (see \cite[Section 2.2.5]{Loday2012}). $BA \otimes_{\pi} A$ is a right dg module over $A$ and there is a quasi-isomorphism $\epsilon_{BA}\otimes \epsilon_{A}: BA \otimes_{\pi} A \to R$ (see \cite[Proposition 2.2.8]{Loday2012} or \cite[Propositon 19.2]{Felix2012}). In general, for a left dg module $(M,\mu_{M})$ over $A$ (where $\mu_{M}: A\otimes_{R} M\to M$ is the left $A$-module action), one can define a left dg comodules over $BA$ which is denoted by $BA\otimes_{\pi}M$ whose underlying graded spaces is $BA \otimes_{R} M$ and whose differential is $$d_{BA}\otimes id_{M} + id_{BA} \otimes d_{M} + (id_{BA} \otimes_{R} \mu_{M})(id_{BA} \otimes_{R} \pi \otimes_{R} id_{M})(\Delta_{BA} \otimes_{R} id_{M}).$$ Let $f:M\to N$ be a homomorphism of left dg $A$-modules, it's easy to check that $BA \otimes_{R} f: BA \otimes_{\pi} M \to BA \otimes_{\pi} N$ is a homomorphism of left dg comodules over $BA$. We get an exact functor $BA\otimes_{\pi}?$ from the category of left dg modules over $A$ to the category of left dg comodules over $BA$.

    Since $BA \otimes_{\pi} A$ is a homotopically  projection resolution of $R$, then $\mathcal{E}nd_{A}(BA \otimes_{\pi} A)$ is a Koszul dual of $A$. By the following lemma, we can take $D(BA)$ as a model of Koszul dual of $A$.
    In \cite{DiMingLu2008}, for an augmented $A_{\infty}$-algebra $A$ over $k$, they define the Koszul dual of $A$ as $D(BA)$. We need remark that since $R$ is a symmetric $k$-algebra, for a graded $R$-bimodule $M$, there is a series of isomorphisms of graded $R$-bimodules $D(M)\cong \mathcal{H}om_{-,R}(M,R) \cong \mathcal{H}om_{R,-}(M,R)$. The first isomorphism $D(M) \cong \mathcal{H}om_{-,R}(M,R)$ will be used in the following lemma.

    \begin{lem}[{\cite[section 19 exercise 4]{Felix2012}}]\label{l3.6}
        For an augmented dg algebra $A$ over $R$, there is a quasi-isomorphism $f\colon D(BA) \to \mathcal{E}nd_{A}(BA \otimes_{\pi} A)$ of dg algebras.
    \end{lem}

    \begin{proof}
        The original statement of this lemma is for $R=k$, but the proof is similar. Here we give a sketch. First since $BA$ is a dg coalgebra and $R$ is a symmetric $k$-algebra, we have $D(BA)\cong \mathcal{H}om_{-,R}(BA,R)$ as dg algebras. $BA \otimes_{\pi} A$ is a left-$\mathcal{H}om_{-,R}(BA,R)$ right-$A$ bimodule. So there is a natural homomorphism of dg algebra $f\colon \mathcal{H}om_{-,R}(BA,R)\to \mathcal{E}nd_{A}(BA \otimes_{\pi} A)$. We need to show $f$ is a quasi-isomorphism. Since $BA \otimes_{\pi} A$ is the augment bar resolution of right dg module $R$ over $A$, there is a quasi-isomorphism $\epsilon_{BA} \otimes_{R} \epsilon \colon BA \otimes_{\pi} A \to R$, where $\epsilon_{BA}$ is the canonical projection $BA \to R$ and $\epsilon$ is the canonical projection $A \to R$. Apply the dg functor $\mathcal{H}om_{A}(BA \otimes_{\pi} A,?)$ to $\epsilon_{BA} \otimes_{R} \epsilon$ we get a quasi-isomorphism
        $\mathcal{E}nd_{A}(BA \otimes_{\pi} A) \to \mathcal{H}om_{A}(BA \otimes_{\pi} A,R)$ of $R$-bimodules.
        There is an exact sequence of right dg module over $A$
        $$ 0 \to BA \otimes_{\pi} \bar{A} \xrightarrow{id\otimes_{R} i} BA \otimes_{\pi} A \xrightarrow{id \otimes_{R}\epsilon} BA \otimes_{\pi} R \to 0.$$ Here $i$ is the inclusion $\bar{A} \to A$ and $BA \otimes_{\pi} R = BA \otimes_{R} R$.
        Apply $\mathcal{H}om_{A}(?,R)$ to the above exact sequence, we get the following exact sequence
        $$ 0 \to \mathcal{H}om_{A}(BA \otimes_{\pi} R,R) \rightarrow \mathcal{H}om_{A}(BA \otimes_{\pi} A ,R) \xrightarrow{(id\otimes_{R} i)^{\ast}}  \mathcal{H}om_{A}(BA \otimes_{\pi} \bar{A},R).$$
        Since the morphism $(id\otimes_{R} i)^{\ast}=0$ , we get $\mathcal{H}om_{A}(BA \otimes_{\pi} A,R)\cong \mathcal{H}om_{A}(BA \otimes_{R} R,R)\cong \mathcal{H}om_{R=A/\bar{A}}(BA \otimes_{R} R,R)=\mathcal{H}om_{-,R}(BA,R)\cong  \mathcal{H}om_{k}(BA,k)$. So there is a quasi-isomorphism $g\colon \mathcal{E}nd_{A}(BA \otimes_{\pi} A) \to \mathcal{H}om_{-,R}(BA,R)$. By directly computation we have $g \circ f =id$, so $f$ is a quasi-isomorphism.
    \end{proof}

    The following lemma is crucial in the proof of our main result.
    \begin{lem}\label{l2}
        Let $Q$ be a positively graded quiver and $A$ a path-$A_{\infty}$-algebra on $Q$. Suppose $U$ is the enveloping algebra of $A$ (see proposition \ref{p2.4}). Then $U$ is quasi-equivalent to $U^{\ast\ast}$.
    \end{lem}

    \begin{proof}
        Since $U$ is the enveloping algebra of $A$, U is an augmented dg algebra over $R:=kQ_{0}$ and the homology $H^{\ast}(U)\cong H^{\ast}(A)=kQ$ is locally finite. By Lemma \ref{l2.5} we only need to show $p\emph{dim} R_{U} < \infty$, $i\emph{dim} R_{U} < \infty$ and $D(A)$ belongs to the smallest triangulated subcategory of $\mathcal{D}(U)$ closed under direct sums and containing $R$.

        Let $V$ denote the graded vector space generated by all arrows in $kQ$ and $i\colon V \to U$ be a section of $R$-bimodule. By a slight abuse of notation, the image of $i$ is also denoted by $V$. We have a projective resolution of $R$ over $kQ$ as follow:
        $$ 0 \to V \otimes_{R} kQ \xrightarrow{\mu} kQ \xrightarrow{p} R \to 0$$
        where $\mu$ is the multiplication of $kQ$ and $p$ is the canonical projection. Because $H^{\ast}U \cong kQ$ as graded algebras, by \cite[Theorem 3.1.c]{Keller1994}, the projective resolution of $R$ yields a homotopically  projective resolution $P \to R$
        such that $P$ admits a filtration right dg $U$-modules $$0 \subseteq F_{0}=U \subseteq F_{1}=P$$ with $P/A \cong (V \otimes_{R} U)[1]$ in the homotopy category of right dg module over $U$. This implies that $p\emph{dim} R < \infty$. Dually using \cite[Theorem 3.2.c]{Keller1994}, one can show that $i\emph{dim} R < \infty$.

        We view $U$ as an augmented $A_{\infty}$-algebra over $kQ_{0}$ and $A$ as a minimal model of $U$. For convenience, we make the following convention.
        We use $\mathcal{D}(U)$ to denote the derived category of right dg modules over $U$ and $\mathcal{D}(A)$ the derived category of right strictly unital $A_{\infty}$-modules over $A$. $Loc(R)$ ($resp$. $Loc(R_{A})$) denote the the smallest triangulated subcategory of $\mathcal{D}(U)$ ($resp$.  $\mathcal{D}(A)$) closed under arbitrary direct sums and containing $R$.

        It remains to show that $D(U) \in Loc(R)$. Since $U$ is a dg algebra and $A$ is a minimal model of $U$, by Proposition \ref{p2.10} and Remark \ref{rmk2.8}, there is a triangulated equivalence between $\mathcal{D}(U)$ and $\mathcal{D}(A)$ which induces a triangulated equivalence between $Loc(R)$ and $Loc(R_{A})$. To show $D(U) \in Loc(R)$ is equivalent to show that $D(U)$, which is viewed as an right $A_{\infty}$-module over $A$, belongs to $Loc(R_{A})$.

        By Proposition \ref{p2.11}, let $(H^{\ast}(D(U)), m_{2},m_{3},\ldots)$ be a minimal model of $D(U)$ as right $A_{\infty}$-module over $A$. Since $H^{\ast}(D(U))$ is isomorphic to $D(kQ)$ as graded $R$-bimodule, $H^{\ast}(D(U))$ is concentrated in nonpositive degrees. Note that $m_{n}$ is of degree $2-n$ for $n \geq 2$, it is easy to see $H^{\ast}(D(U))$ admits a filtration of $A_{\infty}$-submodules $\{F_{p}\}$ for $p \geq 0$ defined as follow: $F_{0}= D(R) \cong R$, $F_{p}=\bigoplus_{n\ge -p} H^{\ast}(D(U))^n$ for $p > 0$. By construction $F_{p+1}/F_{p}$ is concentrated in degree $-p-1$ for $p \geq 0$. Therefore $F_{p+1}/F_{p} \in \mathrm{Add}(R)$ and there is a series of  exact sequences of $A_{\infty}$-modules
        $$ 0 \to F_{p} \xrightarrow{i} F_{p+1} \to F_{p+1}/F_{p} \to 0$$ for $p \geq 0$.
        By \cite[Proposition 5.2]{Keller1999}, we obtain the following series of triangles in $D(A)$:
        $$ F_{p} \to F_{p+1} \to F_{p+1}/F_{p} \to F_{p}[1]$$
        for $p \geq 0$. By induction, all $F_{p} \in Loc(R_{A})$ for $p \geq 0$. Since $\lim\limits_{\longrightarrow} F_{p} = H^{\ast}(D(U))$, we get a triangle
        $$ \bigoplus_{p \geq 0} F_{p} \xrightarrow{\Phi} \bigoplus_{p \geq 0} F_{p} \to H^{\ast}(D(U)) \to (\bigoplus_{p \geq 0} F_{p})[1]$$
        in $D(H^{\ast}(U))$, where $\Phi$ has components
        $F_p\xrightarrow{[1, -i]} F_p\bigoplus F_{p+1}\to \bigoplus_{q\ge0} F_q$. This implies that $H^{\ast}(D(U)) \in Loc(R_{A})$. In addition, $D(U)$ is isomorphic to $H^{\ast}(D(U))$ in $\mathcal{D}(A)$, so $D(U) \in Loc(R_{A})$. Therefore $D(U) \in Loc(R)$.
    \end{proof}

    \begin{thm}\label{t1}
        Let $(A, m_{2}, \ldots)$ be a path $A_{\infty}$-algebra over a positively graded quiver $Q$. Then $A$ is quasi-equivalent to $kQ$ which is viewed as a dg algebra with trivial differential. In particular, the derived category of right strictly unital $A_{\infty}$-modules over $A$ is triangulated equivalent to the derived category of right dg module over $kQ$.
    \end{thm}

    \begin{proof}
        The main technique we use here is Koszul duality for dg algebras. $A$ is an augmented $A_{\infty}$-algebra over $R:=kQ_{0}$, let $U$ be the enveloping algebra of $A$. By definition we need to show that $U$ is quasi-equivalent to $kQ$. By Proposition \ref{p2.4}, there is an $A_{\infty}$-quasi-isomorphism of augmented $A_{\infty}$-algebras $f\colon A \to U$ which preserves augmentations. We choose a `proper' homotopically  projection resolution of $R$ to compute the Koszul dual of $U$. Let $V$ denote the graded vector space generated by all arrows in $kQ$.

        Because $H^{\ast}U \cong H^{\ast}A \cong kQ$ as graded algebras, as proved in Lemma \ref{l2}, there is a homotopically projective resolution $P \to R_{U}$
        such that $P$ admits a filtration of right dg $U$-modules $$0 \subseteq F_{0}=U \subseteq F_{1}=P$$ such that $P/U \cong (V \otimes_{R} U)[1]$.
        We mention that in Keller's construction (see \cite[Theorem 3.2.c]{Keller1994}) one can choose $P$ to be an $R$-bimodule and all morphisms in the above triangle to be homomorphisms of $R$-bimodules. By definition $U^{\ast}=\mathcal{E}nd_{U}(P)$.

        Next we compute the cohomology of $U^{\ast}$. We have the following  triangle
        $$ U \to P \to (V \otimes U)[1] \to U[1] $$ in $\mathcal{D}(U)$, the derived category of right dg modules over $U$. Applying $\mathcal{D}(U)(?,R)$ to this triangle we get a long exact sequence of $R$-bimodules:
        $$
          \xymatrix@C=0.25cm{
          \cdots \ar[r] & \mathcal{D}(U)((V \otimes_{R} U)[n+1], R) \ar[rr] && \mathcal{D}(U)(P[n], R) \ar[rr] && \mathcal{D}(U)(U[n], R) \ar[r] & \cdots }
        $$
        We have
        \begin{equation*}
          \begin{split}
            \mathcal{D}((V \otimes U)[n+1], R) &= H^{-n-1} \mathcal{H}om_{U}(V \otimes U, R) \\
                                               &\cong H^{-n-1} \mathcal{H}om_{R}(V, \mathcal{H}om_{U}(U,R)) \\
                                               &= H^{-n-1} \mathcal{H}om_{R}(V, R)
          \end{split}
        \end{equation*}
        Since $R$ is a dg $U$-module concentrated in degree $0$, we also have $\mathcal{D}_{U}(U[n], R) = H^{-n}(R) = 0$ if $n \neq 0$ and $\mathcal{D}_{U}(U, R) =R$. Together with $D(V) = \mathcal{H}om_{R}(V,R)$ is concentrated in negative degrees with trivial differential, we get that $H^{\ast}(U^{\ast}) \cong \oplus_{n \in \mathds{Z}}\mathcal{D}_{U}(P, R[n]) \cong R \oplus (D(V))[-1]$ as graded $R$-bimodule. By Lemma \ref{l3.6}, $D(BU)$ is quasi-isomorphic to $U^{\ast}$ and $D(BU)$ is augmented, it follows that the minimal model $A'$ of $D(BU)$ is an augmented $A_{\infty}$-algebra with underlying graded $R$-bimodule $R \oplus (D(V))[-1]$.

        Next we prove that the multiplications except $m_{2}$ of $A'$ vanish. By Lemma \ref{l2}, $U$ is quasi-equivalent to $U^{\ast\ast}$, which implies that $H(U^{\ast\ast}) \cong kQ $ as $R$-bimodules. Since $D(B(D(BU)))$ is quasi-isomorphic to $U^{\ast\ast}$ and $D(B(?))$ preserves $A_{\infty}$-quasi-isomorphism, we have $D(BA')$ is quasi-isomorphic to $U^{\ast\ast}$. By direct computation $D(BA') \cong kQ $ as $R$-bimodule (this is because the augmentation ideal of $A'$ is isomorphic to $(D(V)[-1]$ as $R$-bimodule), it forces that the differential of $D(BA')$ vanishes and $D(BA') \cong kQ$ as dg algebra. Combining the facts $U$ is quasi-equivalent to $U^{\ast\ast}$ and $U^{\ast\ast}$ is quasi-isomorphic to $D(BA')$, we obtain that $U$ is quasi-equivalent to $kQ$.
    \end{proof}

    Let $\mathcal{T}$ be a compactly generated (k-linear) algebraic triangulated category. Let $\mathbb{S}$ denote a set of compact generators of $\mathcal{T}$. Assume that $\mathbb{S}$ is a finite set and objects of $\mathbb{S}$ are pairwise non-isomorphic. Let $A_{\mathcal{T}}$ denote the graded algebra structure on the graded vector space $\bigoplus\limits_{n \in \mathbb{Z}}\bigoplus\limits_{X,Y \in \mathbb{S}}Hom_{T}(X, \sum^{n}Y)$ whose multiplication is induced by the composit of morphisms in $\mathcal{T}$.

    As an application of Theorem \ref{t1} , we have the following corollary.

    \begin{cor}
        Let $\mathcal{T}$ be a (k-linear) algebraic triangulated category compactly generated by a finite set of objects.  $A_{\mathcal{T}}$ is defined as above. If $A_{\mathcal{T}}$ is isomorphic to a path algebra of the positively graded quiver Q, then $\mathcal{T}$ is triangulated equivalent to the derived category of $kQ$ viewed as a dg algebra with trivial differential.
    \end{cor}

    \begin{proof}
        By \cite[Theorem 7.6.0.4]{Lefevre-Hasegawa2003}, there is a minimal strictly unital $A_{\infty}$-category (see \cite[Section 7.2]{Keller1999}) $\mathcal{A}$ whose set of objects is $\mathbb{S}$ such that $\mathcal{T}$ is triangulated equivalent to a derived category $D_{\infty}(\mathcal{A})$ of strictly unital $A_{\infty}$-modules over $\mathcal{A}$. By construction, $Hom_{\mathcal{A}}(X,Y)^{i} \cong Hom_{\mathcal{T}}(X,\sum^{n} Y)$ for $X,Y \in \mathbb{S}$. Let $R=\prod_{X\in obj \mathcal{A}} k$ and $\tilde{A}= \bigoplus\limits_{X,Y\in \mathbb{S}} Hom_{\mathcal{A}}(X,Y)$ be the Auslander algebra of $\mathcal{A}$. Then $\tilde{A}$ is a minimal strictly unital $A_{\infty}$-algebra over $R$ and $D_{\infty}(A)$ is triangulated-equivalent to $D_{\infty}(\mathcal{A})$. Since $\tilde{A}\cong A_{\mathcal{T}}$ as graded algebra, $kQ_{0}\cong R$ and $\tilde{A}$ concentrated in nonnegative degrees with 0th degree $R$. Then $\tilde{A}$ is augmented over $R$ and is a path $A_{\infty}$-algebra of positively graded quiver $Q$. Then by Theorem \ref{t1}, $\tilde{A}$ is quasi-equivalent to $A_{\mathcal{T}}$. So $\mathcal{T}$ is triangulated equivalent to the derived category of $A_{\mathcal{T}}$ viewed as a dg algebra with trivial differential.
    \end{proof}

\section*{Acknowledgements}
The author is grateful to Professor Dong Yang for his instructions and suggestions, to Professor Xiao-Wu Chen for his initial idea of this paper and to Professor Yu Ye for the guidance over the past years.

\end{document}